\documentclass[amssymb,amscd,11pt,20pt]{amsart}
\usepackage{amsmath,amsthm}
\usepackage{amsfonts}
\usepackage{amscd}
\usepackage{amscd}
\usepackage[all]{xy}

 \newtheorem{thm}{Theorem}[section]
 \newtheorem{cor}[thm]{Corollary}
 \newtheorem{lem}[thm]{Lemma}
 \newtheorem{prop}[thm]{Proposition}
 \theoremstyle{definition}
 \newtheorem{defn}[thm]{Definition}
 \theoremstyle{remark}
 \newtheorem{rem}[thm]{Remark}
\newtheorem{rems}[thm]{Remarks}
 \newtheorem{ejem}[thm]{Example}
 \newtheorem{ejems}[thm]{Examples}

\numberwithin{equation}{subsection}
\newcommand{\OO}{{\mathcal O}}
\newcommand{\M}{{\mathcal M}}

\newcommand{\U}{{\mathcal U}}

\newcommand{\Nc}{{\mathcal N}}

\newcommand{\ZZ}{{\mathbb Z}}
\newcommand{\Lc}{{\mathcal L}}
\newcommand{\V}{{\mathcal V}}

\DeclareMathOperator{\Aut}{{Aut}}

\newcommand{\enumera}{\begin{enumerate}}
\newcommand{\eenumera}{\end{enumerate}}

\DeclareMathOperator{\Hom}{{Hom}}

\DeclareMathOperator{\Id}{{Id}}

\begin{document}

\title[Homotopy of Ringed Finite Spaces]
 {Homotopy of Ringed Finite Spaces}

\author{ Fernando Sancho de Salas}
\address{Departamento de Matem\'{a}ticas and Instituto Universitario de F�sica Fundamental y Matem�ticas (IUFFyM), Universidad de Salamanca,
Plaza de la Merced 1-4, 37008 Salamanca, Spain}

\email{fsancho@usal.es}



\thanks {The  author was supported by research project MTM2013-45935-P (MINECO)}




\begin{abstract} A ringed finite space is a ringed space whose underlying topological space is finite. The category of ringed finite spaces contains, fully faithfully, the category of finite topological spaces and the category of affine schemes. Any ringed space, endowed with a finite open covering, produces a ringed finite space. We study the homotopy of ringed finite spaces, extending Stong's homotopy classification of finite topological spaces to ringed finite spaces. We also prove that the category of quasi-coherent modules on a ringed finite space is a homotopy invariant.
\end{abstract}

\maketitle

\section*{Introduction}

This paper deals with ringed finite spaces and quasi-coherent modules on them. Let us  motivate why these structures deserve some attention (Theorems 1 and 2 below). Let $S$ be a topological space and let $\U=\{ U_1,\dots, U_n\}$ be a finite covering by open subsets. Let us consider the following equivalence relation on $S$: we say that $s\sim s'$ if $\U$ does not distinguish $s$ and $s'$; that is, if we denote $U^s=\underset{s\in U_i}\cap U_i$, then $s\sim s'$ iff $U^s=U^{s'}$. Let us denote $X=S/\negmedspace\sim$ the quotient set, with the topology given by the following partial order: $[s]\leq [s']$ iff $U^s\supseteq U^{s'}$. $X$ is a finite ($T_0$)-topological space, and the quotient map $\pi\colon S\to X$ is continous.

Assume now that $S$ is a path connected, locally path connected and locally simply connected topological space  and let $\U$ be a finite covering such that the $U^s$ are simply connected. Then (Theorem \ref{fin-sp-assoc-top}):
\medskip

 {\bf Theorem 1.} {\sl  The functors

\[\aligned \left\{\aligned \text{Locally constant sheaves}\\ \text{of abelian groups on $S$}\endaligned \right\} & \overset{\longrightarrow}\leftarrow \left\{ \aligned \text{Locally constant sheaves}\\ \text{of abelian groups on $X$}\endaligned \right\} \\ \M &\to \pi_*\M \\ \pi^*\Nc &\leftarrow \Nc \endaligned \]
are mutually inverse. In other words, $\pi_1(S,s)\to \pi_1(X,\pi(s))$ is an isomorphism between the fundamental groups of $S$ and $X$. Moreover, if the $U^s$ are homotopically trivial, then $\pi\colon S\to X$ is a weak homotopy equivalence, i.e., $\pi_i(S)\to \pi_i(X)$ is an isomorphism for any $i\geq 0$.
}\medskip

Now, if we take the constant sheaf $\ZZ$ on $X$, it turns out that a sheaf of abelian groups on $X$ is locally constant if and only if it is a quasi-coherent $\ZZ$-module (Theorem \ref{qc-fts}). In conclusion, the category of representations of $\pi_1(S)$ on abelian groups is equivalent to the category of quasi-coherent $\ZZ$-modules on the finite space $X$.

Assume now that $S$ is a scheme and that the $U^s$ are affine schemes (a $\U$ with this condition exists if and only if $S$ is quasi-compact and quasi-separated). Let $\OO_S$ be the structural sheaf of $S$ and put $\OO=\pi_*\OO_S$, which is a sheaf of rings on $X$.   Now the result is (Theorem \ref{schemes}):

\medskip
{\bf Theorem 2.} {\sl Let $S$ be a scheme,  $\U$ a finite covering such that the $U^s$ are affine schemes and $(X,\OO)$ the ringed finite space constructed above. The functors \[\aligned \{\text{Quasi-coherent $\OO_S$-modules} \} & \overset{\longrightarrow}\leftarrow \{\text{Quasi-coherent $\OO$-modules} \} \\ \M &\to \pi_*\M \\ \pi^*\Nc &\leftarrow \Nc \endaligned \]
are mutually inverse, i.e., the category of quasi-coherent modules on $S$ is equivalent to the category of quasi-coherent $\OO$-modules on $X$. }\medskip

In \cite{EstradaEnochs} it is proved that the category of quasi-coherent sheaves on a quasi-compact and quasi-separated scheme $S$ is equivalent to the category of quasi-coherent $R$-modules, where $R$ is a ring representation of a finite quiver $\V$. Our point of view is that the quiver $\V$ may be replaced by a finite topological space $X$ and the representation $R$ by a sheaf of rings $\OO_X$. The advantage is that the equivalence between quasi-coherent modules is obtained from a geometric morphism $\pi\colon S\to X$. Thus, this point of view  may be used to prove cohomological results on schemes by proving them on a finite ringed space. For example, one can prove the Theorem of formal functions, Serre's criterion of affineness, flat base change  or Grothendieck's duality in the context of ringed finite spaces (where the proofs are easier)  obtaining those results for schemes as a  particular case. Thus, the  standard hypothesis of separated or semi-separated on schemes may be replaced by the less restrictive hypothesis of quasi-separated. This will be done in a future paper.

In algebraic geometry, quasi-coherent modules and their cohomology play an important role, as locally constant sheaves do in algebraic topology. Theorems 1 and 2 tell us that, under suitable conditions, these structures are determined by a finite model.  All this led us to conclude that it is worthy to  make a  study of ringed finite spaces and of quasi-coherent modules on them.

By a ringed finite space we mean  a ringed space $(X,\OO)$ whose underlying topological space $X$ is finite, i.e. it is a finite topological space endowed with a sheaf $\OO$ of (commutative with unit) rings. It is well known (since Alexandroff) that a finite topological space is equivalent to a finite preordered set, i.e. giving a topology on a finite set is equivalent to giving a preorder relation. Giving a sheaf of rings $\OO$ on a finite topological space $X$ is equivalent to give, for each point $p\in X$, a ring $\OO_p$, and for each $p\leq q$ a morphism of rings $r_{pq}\colon \OO_p\to\OO_q$, satisfying the obvious relations ($r_{pp}=\Id$ for any $p$ and $r_{ql}\circ r_{pq}=r_{pl}$ for any $p\leq q\leq l$). An $\OO$-module $\M$  on $X$ is equivalent to the data: an $\OO_p$-module $\M_p$ for each $p\in X$ and a morphism of $\OO_p$-modules $\M_p\to\M_q$ for each $p\leq q$ (again with the obvious relations).

The category of ringed finite spaces   is a full  subcategory of the category of ringed spaces and it contains (fully faithfully) the category of finite topological spaces  and the category of affine schemes (see Examples \ref{ejemplos}, (1) and (2)). If $(S,\OO_S)$ is an arbitrary ringed space (a topological space, a differentiable manifold, a scheme, etc) and we take a finite covering $\U=\{ U_1,\dots,U_n\}$ by open subsets, there is a natural associated ringed finite space $(X,\OO_X)$ and a morphism of ringed spaces $S\to X$ (see Examples \ref{ejemplos}, (3)). As mentioned above, a particular interesting case is   when $S$ is a quasi-compact and quasi-separated scheme and $\U$ is a locally affine finite covering.

In section 3 we make a  study of the homotopy of ringed finite spaces. We see how the homotopy relation of continuous maps between finite topological spaces can be generalized to morphisms between ringed finite spaces in such a way that Stong's classification (\cite{Stong}) of finite topological  spaces  (via minimal topological spaces)  can be generalized to ringed finite spaces (Theorem \ref{homotopic-classification}). An important fact is that the category of quasi-coherent modules on a ringed finite space is a homotopy invariant: two homotopy equivalent ringed finite spaces have equivalent categories of quasi-coherent sheaves (Theorem \ref{homotinvariance}).

The results of this paper could be formulated in terms of posets and complexes. As in \cite{Barmak}, we have preferred the topological point of view of McCord, Stong and May.

This paper is dedicated to the beloved memory of Prof. Juan Bautista Sancho Guimer{\'a}. I learned from him most of mathematics I know, in particular the use of finite topological spaces in algebraic geometry.

\section{Preliminaries}

In this section we recall elementary facts about finite topological spaces and ringed spaces. The reader may consult \cite{Barmak} for the results on finite topological spaces and \cite{GrothendieckDieudonne} for  ringed spaces.

\subsection{Finite topological spaces}
\medskip

\begin{defn} A finite topological space  is a topological space with a finite number of points.
\end{defn}

Let $X$ be a finite topological space. For each $p\in X$, we shall denote by  $U_p$  the minimum open subset containing  $p$, i.e., the intersection of all the open subsets containing $p$. These $U_p$ form  a minimal base of open subsets.

\begin{defn} A finite preordered set is a finite set with a reflexive and transitive relation (denoted by $\leq$).
\end{defn}

\begin{thm} {\rm (Alexandroff)} There is an equivalence between finite topological spaces and finite preordered sets.

\end{thm}

\begin{proof} If $X$ is a finite topological space, we define the relation: $$p\leq q\quad\text{iff}\quad p\in \bar q \quad (\text{i.e., if } q\in U_p) $$
Conversely, if $X$ is a finite preordered set, we define the following topology on  $X$: the closure of a point $p$ is $\bar p=\{ q\in X: q\leq p\}$.
\end{proof}

\begin{rem} \begin{enumerate}
\item The preorder relation defined above does not coincide with that of \cite{Barmak}, but with its inverse. In other words, the topology associated to a preorder that we have defined above is the dual topology that the one considered in op.cit.
\item If $X$ is a finite topological space, then $U_p=\{ q\in X: p\leq q\}$. Hence $X$ has a minimum $p$ if and only if $X=U_p$.
\end{enumerate}
\end{rem}

A map  $f\colon X\to X'$ between finite topological spaces is continuous if and only if it is monotone: for any $p\leq q$, $f(p)\leq f(q)$.

\begin{prop} A finite topological space is  $T_0$ (i.e., different points have different closures) if and only if the relation $\leq$ is antisymmetric, i.e., $X$ is a partially ordered finite set (a finite poset).
\end{prop}



\medskip
\begin{ejem}\label{covering}{\bf (Finite topological space associated to a finite covering)}. Let $S$ be a topological space and let   $\U=\{U_1,\dots,U_n\}$ be a finite open covering of $S$. Let us consider the following equivalence relation on $S$: we say that $s\sim s'$ if $\U$ does not distinguish $s$ and $s'$, i.e., if we denote $U^s=\underset{s\in U_i}\bigcap U_i$, then $s\sim s'$ iff $U^s=U^{s'}$. Let $X=S/\negmedspace\sim$ be the quotient set with the topology given by the following partial order: $[s]\leq [s']$ iff $U^s\supseteq U^{s'}$.  This is a finite $T_0$-topological space, and the quotient map $\pi\colon S\to X$, $s\mapsto [s]$, is continuous. Indeed, for each $[s]\in X$, one has that $\pi^{-1}(U_{[s]})=U^s$:
\[ s'\in \pi^{-1}(U_{[s]})\Leftrightarrow [s']\geq [s]\Leftrightarrow U^{s'}\subseteq U^s\Leftrightarrow s'\in U^s.\]

We shall say that $X$ is the  finite topological space associated to the topological space $S$ and the finite covering $\U$.

This construction is functorial in $(S,\U)$: Let $f\colon S'\to S$ be a continuous map, $\U$ a finite covering of $S$ and $\U'$ a finite covering of $S'$ that is thinner than $f^{-1}(\U)$ (i.e., for each $s'\in S'$, $U^{s'}\subseteq f^{-1}(U^{f(s')})$). If $\pi\colon S\to X$ and $\pi'\colon S'\to X'$ are the associated finite spaces, one has a continuous map $X'\to X$ and a commutative diagram
\[\xymatrix{ S'\ar[r]^f\ar[d]_{\pi'} & S\ar[d]^\pi\\ X'\ar[r] & X.
}\]
This is an easy consequence of the following:

\begin{lem} $U^{s'_1}\subseteq U^{s'_2}\Rightarrow U^{f(s'_1)}\subseteq U^{f(s'_2)}$.
\end{lem}
\begin{proof}
$U^{s'_1}\subseteq U^{s'_2} \Rightarrow s'_1\in U^{s'_2}\subseteq f^{-1}(U^{f(s'_2)}) \Rightarrow f(s'_1)\in U^{f(s'_2)} \Rightarrow U^{f(s'_1)}\subseteq U^{f(s'_2)}.$
\end{proof}

%

\end{ejem}

%

\subsection{Generalities on ringed spaces}
\medskip

\begin{defn} A ringed space is a pair $(X,\OO)$, where $X$ is a topological space and  $\OO$ is a sheaf of (commutative with unit) rings on  $X$. A morphism or ringed spaces $(X,\OO)\to (X',\OO')$ is a pair $(f,f_\#)$, where $f\colon X\to X'$ is a continuous map and $f_\#\colon \OO'\to f_*\OO$ is a morphism of sheaves of rings (equivalently, a morphism of sheaves of rings $f^\#\colon f^{-1}\OO'\to \OO$).
\end{defn}

\begin{defn} Let $\M$ be an $\OO$-module (a sheaf of $\OO$-modules). We say that $\M$ is {\it quasi-coherent} if for each $x\in X$ there exist an open neighborhood  $U$ of $x$ and an exact sequence
\[ \OO_{\vert U}^I \to \OO_{\vert U}^J\to\M_{\vert U}\to 0\] with $I,J$ arbitrary sets of indexes. Briefly speaking, $\M$ is quasi-coherent if it is locally a cokernel of free modules.

\end{defn}


Let $f\colon X\to Y$ a morphism of ringed spaces. If $\M$ is a quasi-coherent   module  on $Y$, then $f^*\M$ is a quasi-coherent   module on $X$.

\section{Ringed finite spaces}

Let $X$ be a finite topological space. Recall that we have  a preorder relation \[ p\leq q \Leftrightarrow p\in \bar q \Leftrightarrow U_q\subseteq U_p\]

 Giving a sheaf $F$ of abelian groups (resp. rings, etc) on $X$ is equivalent to giving the following data:

 - An abelian group  (resp. a ring, etc) $F_p$ for each $p\in X$.

 - A morphism of groups (resp. rings, etc) $r_{pq}\colon F_p\to F_q$ for each $p\leq q$, satisfying: $r_{pp}=\Id$ for any $p$, and $r_{qr}\circ r_{pq}=r_{pr}$ for any $p\leq q\leq r$. These $r_{pq}$ are called {\it restriction morphisms}.

Indeed, if $F$ is a sheaf on $X$, then  $F_p$ is the stalk of $F$ at $p$, and it coincides with the sections of $F$ on $U_p$. That is
\[ F_p=\text{ stalk of } F \text{ at } p = \text{ sections of } F \text{ on } U_p:=F(U_p)\]
The morphisms $F_p\to F_q$ are just the restriction morphisms  $F(U_p)\to F(U_q)$.

\begin{ejem} Given a group $G$, the constant sheaf $G$ on $X$ is given by the data: $G_p=G$ for any $p\in X$, and $r_{pq}=\Id$ for any $p\leq q$.
\end{ejem}

\begin{defn} A {\it ringed finite space} is a ringed space $(X,\OO )$ such that  $X$ is a finite topological space.
\end{defn}

By the previous consideration, one has a ring $\OO_p$ for each $p\in X$, and a morphism of rings $r_{pq}\colon \OO_p\to\OO_q$ for each $p\leq q$, such that $r_{pp}=\Id$ for any $p\in X$ and  $r_{ql}\circ r_{pq}=r_{pl}$ for any $p\leq q\leq l$.
\medskip

 Giving a morphism of ringed spaces $(X,\OO)\to (X',\OO')$ between two ringed finite spaces, is equivalent to giving:

-  a continuous (i.e. monotone) map $f\colon X\to X'$,

 -  for each  $p\in X$, a ring homomorphism  $f^\#_p\colon \OO'_{f(p)}\to \OO_p$, such that, for any  $p\leq q$, the diagram (denote $p' =f(p), q'=f(q)$)
\[ \xymatrix{ \OO'_{p'} \ar[r]^{f^\#_{p}} \ar[d]_{r_{p'q'}} & \OO_{p}\ar[d]^{r_{pq}}\\ \OO'_{q'} \ar[r]^{f^\#_{q}}   & \OO_{q}}\] is commutative. We shall denote by $\Hom(X,Y)$ the set of morphisms of ringed spaces between two ringed spaces $X$ and $Y$.

\begin{ejems}\label{ejemplos} \item[$\,\,$(1)] {\it Punctual ringed spaces}. A ringed finite space is called punctual if the underlying topological space has only one element.  The sheaf of rings is then just a ring. We shall denote by  $(*,A)$ the ringed finite space with topological space $\{*\}$ and ring $A$. Giving a morphism of ringed spaces  $(X,\OO)\to (*,A)$ is equivalent to giving a ring homomorphism $A\to \OO(X)$. In particular, the category of punctual ringed spaces is equivalent to the (dual) category of rings, i.e., the category of affine schemes. In other words, the category of affine schemes is a full subcategory of the category of ringed finite spaces, precisely the full subcategory of punctual ringed finite spaces.

Any ringed space $(X,\OO)$ has an associated punctual ringed space $(*,\OO(X))$ and a morphism or ringed spaces $\pi\colon (X,\OO)\to (*,\OO(X))$ which is universal for morphisms from $(X,\OO)$ to punctual spaces. In other words, the inclusion functor
\[i\colon \{\text{Punctual ringed spaces}\} \hookrightarrow \{\text{Ringed spaces}\}\] has a left adjoint: $(X,\OO)\mapsto (*,\OO(X))$. For any $\OO(X)$-module $M$, $\pi^*M$ is a quasi-coherent module on $X$. We sometimes denote $\widetilde M:=\pi^*M$.
\medskip

\item[$\,\,$(2)] {\it Finite topological spaces}. Any finite topological space $X$ may be considered as a ringed finite space, taking
the constant sheaf $\ZZ$ as the sheaf of rings. If  $X$ and $Y$ are two finite topological spaces, then giving a morphism of ringed spaces  $(X,\ZZ)\to (Y,\ZZ)$ is just giving a continuous map $X\to Y$. Therefore the category of finite topological spaces is a full subcategory of the category of ringed finite spaces. The (fully faithful) inclusion functor
\[ \aligned \{\text{Finite topological spaces}\} &\hookrightarrow \{\text{Ringed finite spaces} \}\\ X &\mapsto (X,\ZZ)\endaligned\] has a left adjoint, that maps a ringed finite space $(X,\OO)$ to $X$. Of course, this can be done more generally, removing the finiteness hypothesis: the category of topological spaces is a full subcategory of the category of ringed spaces (sending $X$ to $(X,\ZZ)$), and this inclusion has a left adjoint: $(X,\OO)\mapsto X$.
\medskip

\item[$\,\,$(3)] Let $(S,\OO_S)$ be a ringed space (a scheme, a differentiable manifold, an analytic space, ...).
Let $\U=\{U_1,\dots,U_n\}$ be a finite open covering of $S$. Let $X$ be the finite topological space associated to $S$ and $\U$, and $\pi\colon S\to X$ the natural continuous map  (Example \ref{covering}). We have then a sheaf of rings on $X$, namely  $\OO:=\pi_*\OO_S$, so that $\pi\colon (S,\OO_S)\to (X,\OO)$ is a morphism of ringed spaces. We shall say that $(X,\OO)$ is the  {\it ringed finite space associated to the ringed space $S$ and the finite covering $\U$}. This construction is functorial on $(S,\U)$ , as in Example \ref{covering}.

\medskip
\item[$\,\,$(4)] {\it Quasi-compact and quasi-separated schemes}. Let $(S,\OO_S)$ be a scheme and $\U=\{U_1,\dots,U_n\}$  a finite open covering of $S$. We say that $\U$ is {\it locally affine} if for each $s\in S$, the intersection $U^s = \underset{s\in U_i}\cap U_i$ is affine. We have the following:

\begin{prop} Let $(S,\OO_S)$ be a scheme. The following conditions are equivalent:
\enumera
\item $S$ is quasi-compact and quasi-separated.
\item $S$ admits a locally affine finite covering $\U$.
\item There exist a finite topological space $X$ and a continuous map $\pi\colon S\to X$ such that $\pi^{-1}(U_x)$ is affine for any $x\in X$.
\eenumera
\end{prop}

\begin{proof} (1) $\Rightarrow$ (2). Since $S$ is quasi-compact and quasi-separated, we can find  a finite covering  $U_1,\dots, U_n$ of $S$ by affine schemes and a finite covering  $\{ U_{ij}^k\}$  of $U_i\cap U_j$ by affine schemes. Let $\U=\{ U_i, U_{ij}^k\}$ and let us see that it is a locally affine covering of $S$. Let $s\in S$. We have to prove that $U^s$ is affine. If $s$ only belongs to one  $U_i$, then $U^s=U_i$ is affine. If $s$ belongs to more than one $U_i$, let us denote $U_{ij}^s= \underset{s\in U_{ij}^k}\cap U_{ij}^k$. Since $U_{ij}^k$ are affine schemes contained in an affine scheme (for example $U_i$), one has that $U_{ij}^s$ is affine. Now,  $U^s=\underset{i,j}\cap U_{ij}^s$. Put $U^s=U_{i_1j_1}^s\cap \dots \cap U_{i_nj_n}^s$. Replacing each intersection $U_{i_rj_r}^s\cap U_{i_{r+1}j_{r+1}}^s$ by $U_{i_rj_r}^s\cap U_{j_{r}i_{r+1}}^s\cap U_{i_{r+1}j_{r+1}}^s$, we may assume that $j_k=i_{k+1}$, i.e.
\[ U^s=U_{i_1i_2}^s\cap U_{i_2i_3}^s\cap U_{i_3i_4}^s\cap \dots \cap U_{i_{n-1}i_n}^s\]
Now, $U_{i_1i_2}^s\cap U_{i_2i_3}^s$ is affine because it is the intersection of two affine subschemes of the affine scheme $U_{i_2}$. Then $U_{i_1i_2}^s\cap U_{i_2i_3}^s\cap U_{i_3i_4}^s$ is affine because $U_{i_1i_2}^s\cap U_{i_2i_3}^s$ and $U_{i_3i_4}^s$ are affine subschemes of the affine scheme $U_{i_3}$. Proceeding this way, one concludes.

(2) $\Rightarrow$ (3). It suffices to take $X$ as the finite topological space associated to $S$ and $\U$.

(3) $\Rightarrow$ (1). $S$ is covered by the affine open subsets $\{ \pi^{-1}(U_x)\}_{x\in X}$, and the intersections $ \pi^{-1}(U_x)\cap  \pi^{-1}(U_{x'})$ are covered by the affine open subsets $\{ \pi^{-1}(U_y)\}_{y\in U_x\cap U_{x'}}$. Hence $S$ is quasi-compact and quasi-separated.

\end{proof}

\end{ejems}

\subsection{Quasi-coherent modules}

 Let $\M$ be a sheaf of  $\OO$-modules on a ringed finite space $(X,\OO)$. Thus, for each $p\in X$, $\M_p$ is an $\OO_p$-module and for each  $p\leq q$ one has a morphism of $\OO_p$-modules $\M_p\to\M_q$, hence a morphism of  $\OO_q$-modules
\[\M_p\otimes_{\OO_p}\OO_q\to\M_q\]

\begin{rem} From the natural isomorphisms \[\Hom_{\OO_{\vert U_p}}(\OO_{\vert U_p},\M_{\vert U_p})=\Gamma(U_p,\M)=\M_p=\Hom_{\OO_p}(\OO_p,\M_p)\] it follows that, in order to define a morphism of sheaves of modules $\OO_{\vert U_p}\to\M_{\vert U_p}$ it suffices to define a morphism of $\OO_p$-modules $\OO_p\to \M_p$ and this latter is obtained from the former by taking the stalk at $p$.
\end{rem}

\begin{thm}\label{qc} An $\OO$-module $\M$ is quasi-coherent if and only if for any  $p\leq q$ the morphism
\[\M_p\otimes_{\OO_p}\OO_q\to\M_q\]
is an isomorphism.
\end{thm}

\begin{proof} If $\M$ is quasi-coherent, for each point $p$ one has an exact sequence:
\[ \OO_{\vert U_p}^I\to \OO_{\vert U_p}^J \to \M_{\vert U_p} \to 0.\] Taking the stalk at $q\geq p$, one obtains an exact sequence
\[ \OO_q^I\to \OO_q^J \to \M_q \to 0\] On the other hand, tensoring the exact sequence at $p$  by $\otimes_{\OO_p}\OO_q$, yields an exact sequence
\[ \OO_q^I\to \OO_q^J \to \M_p\otimes_{\OO_p}\OO_q \to 0.\] Conclusion follows.

Assume now that  $\M_p\otimes_{\OO_p}\OO_q\to\M_q$ is an isomorphism for any $p\leq q$. We solve $\M_p$ by free $\OO_p$-modules:
\[ \OO_p^I\to \OO_p^J \to \M_p \to 0.\]
 We have then morphisms $\OO_{\vert U_p}^I\to \OO_{\vert U_p}^J \to \M_{\vert U_p}\to 0$. In order to see that this sequence is exact, it suffices to take the stalk at $q\geq p$. Now, the  sequence obtained at  $q$ coincides with the one obtained at $p$ (which is exact) after tensoring by $\otimes_{\OO_p}\OO_q$, hence it is exact.
\end{proof}

\begin{ejem} Let $(X,\OO)$ be a ringed finite space, $A=\OO(X)$ and $\pi\colon (X,\OO)\to (*,A)$ the natural morphism. We know that for any $A$-module $M$, $\widetilde M:=\pi^*M$ is a quasi-coherent module on $X$. The explicit stalkwise description of $\widetilde M$ is given by: $(\widetilde M)_x=M\otimes_A\OO_x$.
\end{ejem}

\begin{cor}\label{corqc} Let $X$ be a ringed finite space with a minimum and $A=\Gamma(X,\OO)$. Then the functors
\[\aligned \{\text{Quasi-coherent $\OO$-modules} \} & \overset{\longrightarrow}\leftarrow \{ \text{$A$-modules}\} \\ \M &\to \Gamma(X,\M) \\ \widetilde M &\leftarrow M \endaligned \]
are mutually inverse.
\end{cor}

\begin{proof} Let $p$ be the minimum of $X$. Then $U_p=X$ and for any sheaf $F$ on $X$, $F_p=\Gamma(X,F)$.
If $\M$ is a quasi-coherent module, then for any $x\in X$, $\M_x=\M_p\otimes_{\OO_p}\OO_x$. That is, $\M$ is univocally determined by its stalk at $p$, i.e., by its global sections.
\end{proof}

This corollary is a particular case of the invariance of the category of quasi-coherent mo\-du\-les under homotopies (see Theorem \ref{homotinvariance}), because any ringed finite space with a minimum $p$ is contractible to $p$ (Remark \ref{contractible}).

\begin{thm}\label{schemes} Let $S$ be a quasi-compact and quasi-separated scheme and   $\U=\{ U_1,\dots, U_n\}$ a locally affine finite  covering. Let $(X,\OO)$ be the finite space associated to $S$ and $\U$,   and $\pi\colon S\to X$ the natural morphism of ringed spaces (see Examples \ref{ejemplos}, (3) and (4)). One has:

1. For any quasi-coherent   $\OO_S$-module $\M$, $\pi_*\M$ is a quasi-coherent $\OO$-module.

2. The functors $\pi^*$ and $\pi_*$ establish an equivalence between the category of quasi-coherent  $\OO_S$-modules and the category of quasi-coherent   $\OO$-modules.

Moreover, for any open subset $U$ of $X$, the morphism $\pi^{-1}(U)\to U$ satisfies 1. and 2.
\end{thm}

\begin{proof} 1. We have to prove that $(\pi_*\M)_p\otimes_{\OO_p}\OO_q\to(\pi_*\M)_q$ is an isomorphism for any $p\leq q$. This is a consequence of the following fact: if $V\subset U$ are open and affine subsets of a scheme $S$ and $\M$ is a quasi-coherent module on $S$, the natural map $\M(U)\otimes_{\OO_S(U)}\OO_S(V)\to\M(V)$ is an isomorphism.

2. Let $\M$ be a quasi-coherent module on $S$. Let us see that the natural map $\pi^*\pi_*\M\to\M$ is an isomorphism. Taking the stalk at $s\in S$, one is reduced to the following fact: if $U$ is an affine open subset of  $S$, then for any $s\in U$ the natural map $\M(U)\otimes_{\OO_S(U)}\OO_{S,s}\to \M_s$ is an isomorphism.

To conclude 2., let $\Nc$ be a quasi-coherent module on $X$ and let us see that the natural map $\Nc\to\pi_*\pi^*\Nc$ is an isomorphism. Taking the stalk at $p\in X$, we have to prove that $\Nc_p\to (\pi^*\Nc)(U)$ is an isomorphism, with $U=\pi^{-1}(U_p)$. Notice that $U$ is an affine open subscheme  and $\OO_S(U)=\OO_p$. Thus, it suffices to prove that, for any $s\in U$, $\Nc_p\otimes_{\OO_p}\OO_{S,s}\to (\pi^*\Nc)(U)\otimes_{\OO_p}\OO_{S,s}$ is an isomorphism. Denoting $q=\pi(s)$, one has that $(\pi^*\Nc)(U)\otimes_{\OO_p}\OO_{S,s}= (\pi^*\Nc)_s=\Nc_q\otimes_{\OO_q}\OO_{S,s}$. Since $\Nc$ is quasi-coherent, $\Nc_q=\Nc_p\otimes_{\OO_p}\OO_q$. Conclusion follows.

Finally, these same proofs work for $\pi\colon \pi^{-1}(U)\to U$, for any open subset $U$ of $X$.
\end{proof}

\begin{thm}\label{qc-fts} Let $X$ be a finite topological space $(\OO=\ZZ$). A sheaf $\M$ of abelian groups on $X$ is quasi-coherent if and only if it is locally constant, i.e., for each $p\in X$, $\M_{\vert U_p}$ is (isomorphic to) a constant sheaf. If $X$ is connected, this means that there exists an abelian group  $G$ such that $\M_{\vert U_p}=G$ for every $p$. If $X$ is not connected, the latter holds in each connected component.
\end{thm}

\begin{proof} Since $\OO$ is the constant sheaf $\ZZ$, the quasi-coherence condition $$``\M_p\otimes_{\OO_p}\OO_q\to\M_q \text{ is an isomorphism}"$$ is equivalent to say that the restriction morphisms $\M_p\to\M_q$ are isomorphisms, i.e., $\M_{\vert U_p}$ is isomorphic to a constant sheaf.
\end{proof}

Now let us prove a topological analog of Theorem \ref{schemes}. First let us recall a basic result about locally constant sheaves and the fundamental group.

\medskip
\noindent{\it Locally constant sheaves and the fundamental group}.
\medskip

Let $S$ be a path connected, locally path connected and locally simply connected topological space and let $\pi_1(S)$ be its fundamental group. Then there is an equivalence between the category of locally constant sheaves on $S$ (with fibre type $G$, an abelian group) and the category of representations of $\pi_1(S)$ on $G$ (i.e., morphisms of groups $\pi_1(S)\to \Aut_{\ZZ-\text{mod.}} G$). In particular, $S$ is simply connected if and only if any locally constant sheaf (of abelian groups) on $S$ is constant.

%
%
%
\medskip

Now, the topological analog of Theorem \ref{schemes} is:

\begin{thm}\label{fin-sp-assoc-top} Let $S$ be a path connected, locally path connected and locally simply connected topological space and  let $\U=\{ U_1,\dots,U_n\}$ be a locally simply connected finite covering of $S$, i.e., for each $s\in S$, the intersection $U^s:=\underset{s\in U_i}\cap U_i$ is simply connected. Let $X$ be the associated finite topological space and $\pi\colon S\to X$ the natural continous map. Then

1. For any locally constant sheaf $\Lc$ on $S$, $\pi_*\Lc$ is a locally constant sheaf on $X$.

2. The functors $\pi^*$ and $\pi_*$ establish an equivalence between the category of locally constant sheaves on $S$ and the category of locally constant sheaves on   $X$. In other words, $\pi_1(S)\to\pi_1(X)$ is an isomorphism.

Moreover, if the $U^s$ are homotopically trivial, then $\pi\colon S\to X$ is a weak homotopy equivalence, i.e., $\pi_i(S)\to\pi_i(X)$ is an isomorphism for any $i$.
\end{thm}

\begin{proof} Let us recall that, on a simply connected space, every locally constant sheaf is constant. Let $x\leq x'$ in $X$, and put $x=\pi(s)$, $x'=\pi(s')$. Then $(\pi_*\Lc)_x \to(\pi_*\Lc)_{x'}$ is the restriction morphism $\Gamma(U^s,\Lc)\to \Gamma(U^{s'},\Lc)$, which is an isomorphism because $\Lc$ is a constant sheaf on $U^s$.

If $\Lc$ is a locally constant sheaf on $S$, the natural morphism $\pi^*\pi_*\Lc\to\Lc$ is an isomorphism, since taking fibre at $s\in S$ one obtains the morphism $\Gamma(U^s,\Lc)\to \Lc_s$, which is an isomorphism because $\Lc$ is a constant sheaf on $U^s$. Finally, if $\Nc$ is a locally constant sheaf on $X$, the natural map $\Nc\to \pi_*\pi^*\Nc$ is an isomorphism: taking fibre at a point $x=\pi(s)$ one obtains the morphism $\Nc_x\to \Gamma(U^s, \pi^*\Nc)$, which is an isomorphism because $\Nc$ is a constant sheaf on $U_x$ (and then $\pi^*\Nc$ is a constant sheaf on $U^s$).

Finally, if the $U^s$ are homotopically trivial, then  $\pi_i(S)\to\pi_i(X)$ is an isomorphism for any $i\geq 0$ by McCord's theorem (see \cite{Barmak}, Theorem 1.4.2).
\end{proof}

\begin{rem} The same proof works for a more general statement: Let $S$ and $T$ be  path connected, locally path connected and locally simply connected topological spaces, $f\colon S\to T$ a continuous map  such that there exists a basis like open (and connected) cover $\U$ of $T$ such that $f^{-1}(U)$ is connected and $\pi_1(f^{-1}(U))\to \pi_1(U)$ is an isomorphism for every $U\in \U$. Then $\pi_1(S)\to\pi_1(T)$ is an isomorphism. It is an analogue of McCord's theorem  for the fundamental group. See also \cite{Quillen}, Proposition 7.6.
\end{rem}

\begin{rems} \begin{enumerate} 
\item Theorems \ref{schemes} and \ref{fin-sp-assoc-top} are not true for non quasi-coherent modules. For example, if $S$ is a homotopically trivial topological space and $\U=\{ S\}$, then the associated finite space is just a point. If $\pi^*$ were an equivalence between the categories of sheaves on $S$ and $X$, this would imply that any sheaf on $S$ is constant. This is not true unless $S$ is a point.
\item Theorems \ref{schemes} and \ref{fin-sp-assoc-top} are not true for non locally affine (resp. locally simply connected) coverings. For example take a scheme $S$ and $\U=\{ S\}$. The associated finite space is just a point. Then $\pi^*$ is an equivalence between quasi-coherent modules if and only if $S$ is affine.
\end{enumerate}
\end{rems}

\section{Homotopy}

For this section, we shall follow the lines of \cite{Barmak}, section 1.3, generalizing them to the ringed case.

Let $(X,\OO_X)$ and $(Y,\OO_Y)$ be two ringed spaces and let $(X\times Y,\OO_{X\times Y})$ the product ringed space: the topological space $X\times Y$ is the ordinary topological product and the sheaf of rings $\OO_{X\times Y}$ is defined as $\OO_{X\times Y}=\pi_X^{-1}\OO_X\otimes_\ZZ\pi_Y^{-1}\OO_Y$, where   $\pi_X$, $\pi_Y$ are the projections of $X\times Y$ onto $X$ and $Y$ respectively.

Let us denote $I=[0,1]$, the unit interval. It is a ringed space (with $\OO_I=\ZZ$). For any ringed space $(X,\OO_X)$, the ringed space $X\times I$ is given by the topological space $X\times I$ and the sheaf of rings $\OO_{X\times I}=\pi_X^{-1}\OO_X$. Then, for any open subsets $U\subseteq X$   and $V\subseteq I$, $\OO_{X\times I}(U\times V)=\OO_X(U)^{\# V}$, where $\# V$ denotes the number of connected components of $V$.

For any $t\in I$, one has a morphism of ringed spaces $i_t\colon X\to X\times I$, defined by the continuous map $i_t(x)=(x,t)$ and the identity morphism of sheaves of rings $i_t^{-1}\OO_{X\times I}=\OO_X\to \OO_X$.

\begin{defn} Let $f,g\colon X\to Y$ be two morphisms of ringed spaces. We say that $f$ and $g$ are {\it homotopy equivalent}, $f\sim g$, if there exists a morphism of ringed spaces $H\colon X\times I\to Y$ such that $H_0=f$ and $H_1=g$ (for any $t\in I$, $H_t\colon X\to Y$ is the composition of $i_t\colon X\to X\times I$ with $H$)
\end{defn}

We can then define the homotopy equivalence between ringed finite spaces:

\begin{defn} Two  ringed finite spaces $X$ and $Y$ are said to be {\it  homotopy equivalent}, denoted by $X\sim Y$, if there exist morphisms
$f\colon X\to Y$ and $g\colon Y\to X$ such that $g\circ f \sim \Id_X$ and $f\circ g\sim \Id_Y$.
\end{defn}

Let $f,g\colon X\to Y$ be two morphisms of ringed spaces, $S$ a subspace of $X$. We leave the reader to define the notion of being homotopic relative to $S$ and hence the notion of a strong deformation retract.

\subsection{Homotopy of ringed finite spaces}

Let us now reduce to ringed finite spaces. Let $X$, $Y$ be  finite topological spaces and $\Hom(X,Y)$ the set of continuous maps, which is a finite set. This set has a preorder, the pointwise preorder:
\[ f\leq g \iff f(x)\leq g(x) \text{ for any } x\in X,\]
hence $\Hom(X,Y)$ is a finite topological space.

It is easy to prove that two continuous maps $f,g\colon X\to Y$ are homotopy equivalent  if and only if they belong to the same connected component of $\Hom(X,Y)$. In other words, if we denote $f\equiv g$ if either $f\leq g$ or $f\geq g$, then  $f\sim g$ if and only if there exists a sequence
\[ f=f_0\equiv f_1 \equiv \cdots \equiv f_n=g,\qquad f_i\in\Hom(X,Y)\]

Assume now that $X$ and $Y$ are ringed finite spaces and $\Hom(X,Y)$ is the set of morphisms of ringed spaces. It is no longer a finite set, however we can define a preorder relation:

\begin{defn} Let $f,g\colon X\to Y$ be two morphisms of ringed spaces. We say that $f\leq g$ if:

(1) $f(x)\leq g(x)$ for any $x\in X$.

(2) For any $x\in X$ the triangle
\[ \xymatrix{\OO_{f(x)}\ar[rr]^{r_{f(x)g(x)}}\ar[rd]_{f^\#_x} & & \OO_{g(x)}\ar[ld]^{g^\#_x}\\ & \OO_x & }\] is commutative. We shall denote by $f\equiv g$ if either $f\leq g$ or $f\geq g$.
\end{defn}

\begin{rems} \label{rem} 

(a) Condition (1) is equivalent to say that for any open subset $V$ of $Y$, one has $f^{-1}(V)\subseteq g^{-1}(V)$. Thus, for any sheaf $F$ on $X$, one has the restriction morphism $F(g^{-1}(V))\to F(f^{-1}(V))$, i.e., a morphism of sheaves $ g_*F\to f_*F$. By adjunction, one has, for any sheaf $G$ on $Y$, a morphism of sheaves $f^{-1}G\to g^{-1}G$,  whose stalkwise description at a point $x$ is just the restriction morphism $r_{f(x)g(x)}\colon G_{f(x)}\to G_{g(x)}$. Thus, condition (2) is equivalent to say that the triangle

\[ \xymatrix{f^{-1}\OO_Y\ar[rr] \ar[rd]_{f^\#} & & g^{-1}\OO_Y\ar[ld]^{g^\#}\\ & \OO_X & }\] is commutative, or equivalently, that the diagram  \[ \xymatrix{g_*\OO_X\ar[rr]  & & f_*\OO_X \\ & \OO_Y\ar[ul]^{g_\#} \ar[ur]_{f_\#} & }\] is commutative.

(b) \label{rem} If $f(x)=g(x)$ for any $x\in X$ (i.e., $f$ and $g$ coincide as continuous maps) and $f\leq g$, then $f=g$.
\end{rems}

\begin{prop} Let $f,g\colon X\to Y$ be two morphisms of ringed finite spaces. Then $f$ and $g$ are  homotopy equivalent  if and only if there exists a sequence:
\[ f=f_0\equiv f_1 \equiv \cdots \equiv f_n=g,\qquad f_i\in\Hom(X,Y) \]
\end{prop}

\begin{proof} It is a consequence of the following lemmas.
\end{proof}

\begin{lem} Let $f,g\colon X\to Y$ be two morphisms between ringed finite spaces. If $f\leq g$, then $f$ is homotopy equivalent to $g$.
\end{lem}

\begin{proof} Let $H\colon X\times I\to Y$ be the map defined by $$H(x,t)=\left\{ \aligned f(x),&\text{ for } t=0 \\ g(x),&\text{ for }t>0\endaligned \right. .$$ For any $y\in Y$, $f^{-1}(U_y)\subseteq g^{-1}(U_y)$, because $f(x)\leq g(x)$ for any $x\in X$. It follows that
$$H^{-1}(U_y)=(f^{-1}(U_y)\times I)\cup (g^{-1}(U_y)\times (0,1]).$$ Thus $H$ is continuous. Moreover, one has the exact sequence
\[ 0\to \OO_{X\times I}(H^{-1}(U_y))\to  \OO_{X\times I} (f^{-1}(U_y)\times I)\times \OO_{X\times I}(g^{-1}(U_y)\times (0,1]) \to  \OO_{X\times I}(f^{-1}(U_y)\times (0,1]),\] i.e., an exact sequence
\[ 0\to H_*\OO_{X\times I}  \to f_*\OO_X \times  g_*\OO_X \to  f_*\OO_X.\]
By Remark \ref{rem}, (a), one obtains a morphism $\OO_Y\to H_*\OO_{X\times I}$. Thus $H$ is a morphism of ringed spaces, and $H_0=f$, $H_1=g$.

\end{proof}

\begin{lem} Let $H\colon X\times I\to Y$ be a morphism of ringed spaces such that $H(x,t)=H(x,t')$ for any $t,t'>0$. Then $H_0\leq H_1$.
\end{lem}

\begin{proof} Let us denote $f=H_0$, $g=H_1$.

1) $f(x)\leq g(x)$ for any $x\in X$. Let $y=f(x)$. Since $H$ is continuous, there exists $\epsilon >0$ such that $H(x,t)\in U_y$ for any $t<\epsilon$. Thus $g(x)=H_t(x)\in U_y$, i.e., $g(x)\geq f(x)$.

2) For any $y\in Y$,  $H^{-1}(U_y)$ is the union of $(f^{-1}(U_y)\times I)$ and $(g^{-1}(U_y)\times (0,1])$, whose intersection is    $f^{-1}(U_y)\times (0,1]$. From the commutative diagram
\[ \xymatrix{\OO_{X\times I}(H^{-1}(U_y))\ar[r]\ar[d] &  \OO_{X\times I}(f^{-1}(U_y)\times I)=\OO_X(f^{-1}(U_y))\ar[d]^{\Id} \\ \OO_X(g^{-1}(U_y))=\OO_{X\times I}(g^{-1}(U_y)\times (0,1])\ar[r] & \OO_{X\times I}(f^{-1}(U_y)\times (0,1])=\OO_X(f^{-1}(U_y))   }\] one obtains a commutative diagram
\[ \xymatrix{ H_*\OO_{X\times I}\ar[rr] \ar[rd]  & & f_*\OO_X \\ &  g_*\OO_X   \ar[ur]  & }\]
and composing with the morphism $\OO_Y\to H_*\OO_{X\times I}$, yields a commutative diagram
\[ \xymatrix{ \OO_Y\ar[rr]^{f_\#} \ar[rd]_{g_\#}  & & f_*\OO_X \\ &  g_*\OO_X   \ar[ur]  & }\]

With all, $f\leq g$.

\end{proof}

\begin{rem}\label{contractible} Any ringed finite space $X$ with a minimum $p$ is contractible to $p$, i.e. it is homotopy equivalent to the punctual ringed space $(p,\OO_p)$. Indeed, one has a natural morphism $i_p\colon (p,\OO_p)\to X$. On the other hand, since $p$ is the minimum, $X=U_p$ and $\OO_p=\Gamma(X,\OO_X)$, and we have the natural morphism (see Examples \ref{ejemplos}, (1)) $\pi\colon X\to (p,\OO_p)$. The composition $\pi\circ i_p$ is the identity and   $i_p\circ \pi \geq \Id_X$.
\end{rem}

\begin{prop}\label{homotinvarianceProp} Let $f,g\colon X\to Y$ be two morphisms of ringed finite spaces. If $f\sim g$, then, for any quasi-coherent sheaf $\M$ on $Y$, one has $f^*\M=g^*\M$.
\end{prop}

\begin{proof} We may assume that $f\leq g$. Then, for any $x\in X$, $$(f^*\M)_x=\M_{f(x)}\otimes_{\OO_{f(x)}}\OO_x = \M_{f(x)}\otimes_{\OO_{f(x)}}\OO_{g(x)}\otimes_{\OO_{g(x)}}\OO_x = \M_{g(x)}\otimes_{\OO_{g(x)}}\OO_x =(g^*\M)_x$$ where the second equality is due to the hypothesis $f\leq g$ and the third one to the quasi-coherence of $\M$.
\end{proof}

\begin{rems} \begin{enumerate} \item Proposition \ref{homotinvarianceProp} is not true if $\M$ is not quasi-coherent. For example, let $X$ be a finite topological space with a minimum $p$. Then $X$ is contractible to $p$, i.e., the identity $\Id\colon X\to X$ is homotopic to the constant map $g\colon X\to X$, $g(x)=p$. If $\M$ is a non constant sheaf on $X$, then $\Id^*\M$ is not equal to $g^*\M$ (they  are not even isomorphic), since $g^*\M$ is a constant sheaf.
\item We do not know if Proposition \ref{homotinvarianceProp} holds for general ringed spaces. 
\end{enumerate}
\end{rems}

The following theorem is now straightforward (and it generalizes Corollary \ref{corqc}):

\begin{thm}\label{homotinvariance} If $X$ and $Y$ are homotopy equivalent ringed finite spaces, then their categories of quasi-coherent modules are equivalent. In other words, the category of quasi-coherent modules on a ringed finite space is a homotopy invariant.
\end{thm}

\begin{rem} We do not know if this theorem holds for general (non finite) ringed spaces.
\end{rem}

\subsection{Homotopy classification: minimal spaces}

Here we see that Stong's homotopical classification of finite topological spaces via minimal topological spaces (\cite{Stong}) can be reproduced in the ringed context.

First of all, let us prove that any ringed finite space is homotopy equivalent to its $T_0$-associated space. Let $X$ be a ringed finite space, $X_0$ its associated $T_0$-space and $\pi\colon X\to X_0$ the quotient map. Let us denote $\OO_0=\pi_*\OO$. Then $(X,\OO)\to (X_0,\OO_0)$ is a morphism of ringed spaces.  The preimage $\pi^{-1}$ gives a bijection between the open subsets of $X_0$ and the open subsets of $X$. Hence, for any $x\in X$, $\OO_x={\OO_0}_{\pi(x)}$, and any section $s\colon X_0\to X$ of $\pi$ is continuous and a morphism of ringed spaces. The composition $\pi\circ s$ is the identity and the composition $s\circ \pi$ is homotopic to the identity, because $\OO_x=\OO_{s(\pi(x))}$. We have then proved:

\begin{prop}\label{sdr} $(X_0,\OO_0) \hookrightarrow (X,\OO_X)$ is a strong deformation retract.
\end{prop}

Let $X$ be a ringed finite $T_0$-space. Let us generalize the notions of up beat point and down beat point to the ringed case.

\begin{defn} A point $p\in X$ is called a {\it down beat point} if  $\bar p -\{ p\}$ has a maximum. A point $p$ is called an {\it up beat point} if $U_p- \{ p\}$ has a minimum $q$ and $r_{pq}\colon \OO_p\to\OO_q$ is an isomorphism. In any of these cases we say that $p$ is a {\it beat point} of $X$.
\end{defn}

\begin{prop}\label{beating} Let $X$ be a ringed finite $T_0$-space and $p\in X$ a beat point. Then $X- \{ p\}$ is a strong deformation retract of $X$.
\end{prop}

\begin{proof} Assume that $p$ is a down beat point and let $q$ be the maximum of $\bar p- \{ p\}$. Define the retraction $r\colon X\to X-\{ p\}$ by $r(p)=q$. It is clearly continuous (order preserving). It is a ringed morphism because one has the restriction morphism $\OO_q\to\OO_p$. If $i\colon X-\{ p\}\hookrightarrow X$ is the inclusion, then $i\circ r\leq \Id_X$ and we are done.

Assume now that $p$ is an up beat point and let $q$ be the minimum of $U_p- \{ p\}$. Define the retraction $r\colon X\to X-\{ p\}$ by $r(p)=q$. It is order preserving, hence continuous. By hypothesis the restriction morphism $\OO_p\to\OO_q$ is an isomorphism, so that $r$ is a morphism of ringed spaces. Finally, $i\circ r\geq \Id_X$ and we are done.
\end{proof}

\begin{defn}  A ringed finite $T_0$-space is a {\it minimal} ringed finite space if it has no beat points. A {\it core} of a ringed finite space $X$ is a strong deformation retract which is a minimal ringed finite space.
\end{defn}

By Propositions \ref{sdr} and \ref{beating} we deduce that every ringed finite space
has a core. Given a ringed finite space $X$, one can find a $T_0$-strong deformation
retract $X_0\subseteq X$ and then remove beat points one by one to obtain a minimal
ringed finite space. As in the topological case, the notable property about this construction is that in fact the
core of a ringed finite space is unique up to isomorphism, moreover: two ringed finite spaces are homotopy equivalent if and only if their cores are isomorphic.

\begin{thm}\label{minimal} Let $ X$ be a minimal ringed finite space. A map $f \colon X \to X$ is
homotopic to the identity if and only if $f = \Id_X$.
\end{thm}

\begin{proof}  We may suppose that $ f \leq \Id_X$ or $f \geq \Id_X$. Assume
$ f \leq \Id_X$. By Remark \ref{rem}, (b), it suffices to prove that $f(x)=x$ for any $x\in X$. On the contrary,  let $p\in X$ be minimal with the condition $f(x)\neq x$. Hence $f(p)<p$ and $f(x)=x$ for any $x<p$. Then $f(p)$ is the maximum of $\bar p-\{ p\}$, which contradicts that $X$ has no down beat points.

Assume now that $f\geq \Id_X$. Again, it suffices to prove that $f(x)=x$ for any $x\in X$. On the contrary, let  $p\in X$ be maximal with the condition $f(x)\neq x$. Then $f(p)>p$ and $f(x)=x$ for any $x>p$. Hence $q=f(p)$ is the minimum of $U_p-\{ p\}$. Moreover  $f$ is a morphism of ringed spaces, hence it gives a commutative diagram
\[ \xymatrix {\OO_q=\OO_{f(p)} \ar[r]^{\quad f^\#_p}\ar[d]_{\Id} &\OO_p \ar[d]^{r_{pq}}\\ \OO_q=\OO_{f(q)}\ar[r]^{\quad f^\#_q} & \OO_q.}\]  Moreover, since $f\geq \Id_X$, the triangles
\[ \xymatrix{\OO_{p}\ar[rr]^{r_{pq}}\ar[rd]_{\Id^\#_p} & & \OO_{q}\ar[ld]^{f^\#_p}\\ & \OO_p & }\quad
\xymatrix{\OO_{q}\ar[rr]^{r_{qq}}\ar[rd]_{\Id^\#_q} & & \OO_{q}\ar[ld]^{f^\#_q}\\ & \OO_q & }\] are commutative. One concludes that   $r_{pq}$ is an isomorphism and $p$ is an up beat point of $X$.
\end{proof}

\begin{thm}\label{homotopic-classification} (Classification Theorem). A homotopy equivalence between
minimal ringed finite spaces is an isomorphism. In particular the core of a ringed
finite space is unique up to isomorphism and two ringed finite spaces are homotopy
equivalent if and only if they have isomorphic cores.
\end{thm}

\begin{proof}  Let $f \colon X \to Y$ be a homotopy equivalence between minimal ringed finite spaces and let $g \colon Y \to X$ be a homotopy inverse. Then $gf = \Id_X$ and $fg = \Id_Y$ by
Theorem \ref{minimal}. Thus, f is an isomorphism. If $X_1$ and $X_2$ are two cores of
a ringed finite space $X$, then they are homotopy equivalent minimal ringed finite spaces, and therefore, isomorphic. Two ringed finite spaces $X$ and $Y$ have the same
homotopy type if and only if their cores are homotopy equivalent, but this is
the case only if they are isomorphic.
\end{proof}

\end{document}